\documentclass[12pt]{amsart}

\author{Paul \textsc{Poncet}}
\address{CMAP, \'{E}cole Polytechnique, Route de Saclay, 91128 Palaiseau Cedex, France \\
and INRIA, Saclay--\^{I}le-de-France}

\email{poncet@cmap.polytechnique.fr}

\usepackage{stmaryrd}
\usepackage{amssymb}
\usepackage{amsmath}
\usepackage{t1enc}
\usepackage{yhmath}
\usepackage{calrsfs} 
\usepackage{times} 

\newcommand{\reels}{\mathbb R}
\newcommand{\ddint}{\int^{\scriptscriptstyle\infty}\!\!}
\newcommand{\dint}[1]{\int^{\scriptscriptstyle\infty}_{#1}\!\!}

\newtheorem{theorem}{Theorem}[section]
\newtheorem{corollary}[theorem]{Corollary}
\newtheorem{proposition}[theorem]{Proposition}
\newtheorem{lemma}[theorem]{Lemma}

\theoremstyle{definition}
\newtheorem{definition}[theorem]{Definition}

\newenvironment{acknowledgements}[1][]{\par\vspace{0.5cm}\noindent\textbf{Acknowledgements#1.} }{\par}

\begin{document}

\title{The idempotent Radon--Nikodym theorem \\ has a converse statement}

\date{\today}



\begin{abstract}
Idempotent integration is an analogue of the Lebesgue integration where $\sigma$-additive measures are replaced by $\sigma$-maxitive measures. 
It has proved useful in many areas of mathematics such as fuzzy set theory, optimization, idempotent analysis, large deviation theory, or extreme value theory.
Existence of Radon--Nikodym derivatives, which turns out to be crucial in all of these applications, was proved by Sugeno and Murofushi. Here we show a converse statement to this idempotent version of the Radon--Nikodym theorem, i.e.\ we characterize the $\sigma$-maxitive measures that have the Radon--Nikodym property. 
\end{abstract}

\maketitle

\section{Introduction}

Maxitive measures, originally introduced by Shilkret \cite{Shilkret71}, are defined analogously to classical finitely additive measures or charges with the supremum operation, denoted $\oplus$, in place of the addition $+$. 
More precisely, 
a \textit{maxitive measure} on a $\sigma$-algebra $\mathrsfs{B}$ is a map $\nu : \mathrsfs{B} \rightarrow \overline{\reels}_+$ such that $\nu(\emptyset) = 0$ and 
$$ 
\nu(B_1 \cup B_2) = \nu(B_1) \oplus \nu(B_2), 
$$
for all $B_1, B_2 \in \mathrsfs{B}$. 
It is \textit{$\sigma$-maxitive} if it commutes with unions of nondecreasing sequences of elements of $\mathrsfs{B}$. One should note that a $\sigma$-maxitive measure does not necessarily commute with \textit{intersections} of nonincreasing sequences, unlike $\sigma$-additive measures. 

A corresponding ``maxitive'' integral, paralleling Lebesgue's integration theory, was built by Shilkret. 
It was rediscovered independently and generalized by Sugeno and Murofushi \cite{Sugeno87} and by Maslov \cite{Maslov87}. 
Since then, this integral has been studied and used 
by several authors with motivations from dimension theory and fractal geometry, optimization, capacities and large deviations of random processes, fuzzy sets and possibility theory, idempotent analysis and max-plus (tropical) algebra. 

Because of the numerous fields of application just listed, the wording around maxitive measures is not unique. 
For instance, Maslov coined the term \textit{idempotent integration}, which is also of wide use. 
Notations may also diverge; we adopt the choice of Gerritse \cite{Gerritse96} and write 
$
\dint{B} f \cdot d\nu 
$ 
for the Shilkret integral of a measurable map $f$ with respect to the maxitive measure $\nu$ on a measurable set $B$. 
The index $\infty$ is not an integration bound, it recalls the fact that the Shilkret integral can be seen as a limit of a sequence of Choquet integrals. 

More generally, we shall consider the \textit{idempotent $\odot$-integral} 
$$
\dint{B} f \odot d\nu, 
$$ 
where $\odot$ is a \textit{pseudo-multiplication}, i.e.\ a binary relation satisfying a series of natural properties. If $\odot$ is the usual multiplication (resp.\ the minimum $\wedge$), then the idempotent $\odot$-integral specializes to the Shilkret (resp.\ Sugeno) integral.  

In all of the fields of application listed above, a Radon--Nikodym like theorem is often essential. For instance, a comprehensive theory of possibilities (where a \textit{possibility measure} is the maxitive analogue of a probability measure) cannot do without a notion of \textit{conditional} possibility (just like one needs that of conditional expected value in probability theory). Its existence happens to be ensured by that of Radon--Nikodym derivatives (or \textit{densities}).  
Such a theorem is actually available: it was proved in \cite{Sugeno87} by Sugeno and Murofushi. 
These authors showed that, if $\nu$ and $\tau$ are $\sigma$-maxitive measures on a $\sigma$-algebra $\mathrsfs{B}$, with $\tau$ $\sigma$-$\odot$-finite and $\sigma$-principal,  then $\nu$ is $\odot$-absolutely continuous with respect to $\tau$ if and only if there exists some $\mathrsfs{B}$-measurable map $c : E \rightarrow \overline{\reels}_+$ such that 
$$
\nu(B) = \dint{B} c \odot d\tau, 
$$
for all $B\in \mathrsfs{B}$. 

Given that $\sigma$-$\odot$-finiteness and $\odot$-absolute continuity generalize the usual concepts of $\sigma$-finiteness and absolute continuity to the setting of the pseudo-multiplication $\odot$, the assertion looks like the classical Radon--Nikodym theorem, except that one needs an unusual condition on the dominating measure $\tau$, namely \textit{$\sigma$-principality}.  
This condition roughly says that every $\sigma$-ideal of $\mathrsfs{B}$ has a  greatest element ``modulo negligible sets''. Although $\sigma$-finite  $\sigma$-additive measures are always $\sigma$-principal, this is not true for $\sigma$-finite $\sigma$-maxitive measures. For instance, every $\sigma$-maxitive measure $\nu$ is $\odot$-absolutely continuous with respect to the $\sigma$-maxitive measure $\delta_{\#}$, defined on the same $\sigma$-algebra $\mathrsfs{B}$ by $\delta_{\#}(B) = 1$ if $B$ is nonempty and $\delta_{\#}(\emptyset) = 0$; however, $\nu$ does not always have a density with respect to $\delta_{\#}$. 

After the article \cite{Sugeno87}, many authors have published results of Radon--Nikodym flavour for maxitive measures. This is the case of Agbeko \cite{Agbeko95}, Akian \cite{Akian99}, Barron et al.\ \cite{Barron00}, and Drewnowski \cite{Drewnowski09}. In some cases, the authors were not aware of the existence of idempotent integration theory. In \cite{Poncet11}, we explained why these results are already encompassed in the Sugeno--Murofushi theorem. 

The purpose of this paper is to prove a converse to the Sugeno--Murofushi theorem. A $\sigma$-maxitive measure $\tau$ has the \textit{Radon--Nikodym property} if every $\sigma$-maxitive measure $\odot$-dominated by $\tau$ has a density with respect to $\tau$. Put together with the Sugeno--Murofushi theorem, our main result is the following:

\begin{theorem}
Given a non-degenerate pseudo-multiplication $\odot$, a $\sigma$-max\-itive measure satisfies the Radon--Nikodym property with respect to the idempotent $\odot$-integral if and only if it is $\sigma$-$\odot$-finite and $\sigma$-principal. 
\end{theorem}

This result ensures the minimality of the conditions of $\sigma$-$\odot$-finiteness and $\sigma$-principality. 
We shall prove it with the help of the ``quotient space'' associated with the $\sigma$-maxitive measure, i.e.\ we shall get rid of negligible sets by an appropriate equivalence relation. 
Such a characterization will be useful in a future work to try to investigate \textit{spaces} (like modules over the idempotent semifield $\reels^{\max}_+ = (\reels_+, \oplus, \times)$) with the Radon--Nikodym property; see the discussion in \cite{Poncet12b}. 

The paper is organized as follows. 
We introduce pseudo-multiplications $\odot$ and some of their properties in Section~\ref{secmax}; we also recall the notions of $\sigma$-maxitive measure and idempotent $\odot$-integral and their basic properties. 
In Section~\ref{secRN} we recall the Radon--Nikodym type theorem for the idempotent $\odot$-integral proved by Sugeno and Murofushi. 
In Section~\ref{sec:quotient} we define the quotient space associated with a $\sigma$-maxitive measure and characterize maxitive measures satisfying the Radon--Nikodym property. 
The usual multiplication $\times$ and the minimum $\wedge$ are particular cases of the general binary relation $\odot$, so our main result specializes to both the Shilkret and the Sugeno integrals. 


\section{Preliminaries on maxitives measures and idempotent integration}\label{secmax}

\subsection{Pseudo-multiplications and their properties}\label{par:odot}

In this paper, we consider a binary relation $\odot$ defined on $\overline{\reels}_+ \times \overline{\reels}_+$ with the following properties: 
\begin{itemize}
	\item associativity; 
	\item continuity on $(0, \infty) \times [0, \infty]$; 
	\item continuity of the map $s \mapsto s \odot t$ on $(0, \infty]$, for all $t$; 
	\item monotonicity in both components; 
	\item existence of a left identity element $1_{\odot}$, i.e.\ $1_{\odot} \odot t =  t$ for all $t$; 
	\item absence of zero divisors, i.e.\ $s \odot t = 0$ implies $0 \in  \{s, t\}$, for all $s, t$; 
	\item $0$ is an annihilator, i.e.\ $0 \odot t = t \odot 0 = 0$, for all $t$. 
\end{itemize}
We call such a $\odot$ a \textit{pseudo-multiplication}. Note that the axioms above are stronger than in \cite{Sugeno87}, where associativity is not assumed. 

We consider the map $O : \overline{\reels}_+ \rightarrow \overline{\reels}_+$ defined by $O(t) = \inf_{s > 0} s \odot t$. 
An element $t$ of $\overline{\reels}_+$ is \textit{$\odot$-finite} if $O(t) = 0$ (and $t$ is \textit{$\odot$-infinite} otherwise). 
We conventionally write $t \ll_{\odot} \infty$ for a $\odot$-finite element $t$. 
If $O(1_{\odot}) = 0$, we say that the pseudo-multiplication $\odot$ is \textit{non-degenerate}. This amounts to say that the set of $\odot$-finite elements differs from $\{0\}$. 

In what follows, we shall use the three following results, the proofs of which are given in a separate note (see Poncet \cite{Poncet13a}). 

\begin{lemma}\label{lem:inv}
Given a non-degenerate pseudo-multiplication $\odot$, the following conditions are equivalent for an element $t \in \overline{\reels}_+$: 
\begin{itemize}
	\item $t$ is $\odot$-finite;
	\item $s \odot t \ll_{\odot} \infty$ for some $s > 0$; 	
	\item $s \odot t \leqslant 1_{\odot}$ for some $s > 0$; 
	\item $t \odot s' \leqslant 1_{\odot}$ for some $s' > 0$; 
	\item $t \odot s' \ll_{\odot} \infty$ for some $s' > 0$.  	
\end{itemize}
\end{lemma}

\begin{theorem}\label{lem:phi}
Given a pseudo-multiplication $\odot$, the following conditions are equivalent: 
\begin{itemize}
	\item $\odot$ is non-degenerate, i.e.\ $1_{\odot}$ is $\odot$-finite; 
	\item there exists some positive $\odot$-finite element; 
	\item the monoid $([0, 1_{\odot}], \odot)$ is commutative; 
	\item the set $F_{\odot}$ of $\odot$-finite elements is either $[0, \infty]$ or of the form $[0, \phi)$ for some $\phi \in (1_{\odot}, \infty]$. 
\end{itemize}
Moreover, in the case where $F_{\odot} = [0, \phi)$, 
then $\phi$ satisfies $O(\phi) = \phi$ and $t \odot \phi = \phi \odot t = \phi$, for all $0 < t \leqslant \phi$. In particular, $\phi$ is idempotent, i.e.\ $\phi \odot \phi = \phi$. 
\end{theorem}

\begin{corollary}\label{lem:dec}
Given a pseudo-multiplication $\odot$, it is not possible to find $t < \phi$ and $t' > \phi$ such that $t \odot t' = \phi$, if $\phi$ denotes the supremum of the set of $\odot$-finite elements. 
\end{corollary}

\subsection{Definition of maxitive measures} 

Let $E$ be a nonempty set. 
A collection $\mathrsfs{B}$ of subsets of $E$ containing $E$, closed under countable unions and the formation of complements is a \textit{$\sigma$-algebra}. 
A \textit{$\sigma$-ideal} of a $\sigma$-algebra $\mathrsfs{B}$ is a nonempty subset $\mathrsfs{I}$ of $\mathrsfs{B}$ that is closed under countable unions and such that $A \subset B \in \mathrsfs{I}$ and $A \in \mathrsfs{B}$ imply $A \in \mathrsfs{I}$. 

Assume in all the sequel that $\mathrsfs{B}$ is a $\sigma$-algebra on $E$. 
A \textit{set function} on $\mathrsfs{B}$ is a map $\mu : \mathrsfs{B} \rightarrow \overline{\reels}_+$ equal to zero at the empty set. A set function $\mu$ is 
 \textit{monotone} if $\mu(B) \leqslant \mu(B')$ for all $B, B' \in \mathrsfs{B}$ such that $B \subset B'$. A monotone set function $\mu$ is \textit{$\odot$-finite} if $\mu(E) \ll_{\odot} \infty$, and \textit{$\sigma$-$\odot$-finite} if there exists some countable family $\{B_n\}_{n \in \mathbb{N}}$ of elements of $\mathrsfs{B}$ covering $E$ such that $\mu(B_n) \ll_{\odot} \infty$ for all $n$; a subset $N$ of $E$ is $\mu$-\textit{negligible} if it is contained in some $B \in \mathrsfs{B}$ such that $\mu(B) = 0$. 



A \textit{maxitive} (resp.\ \textit{$\sigma$-maxitive}) \textit{measure} on $\mathrsfs{B}$ is a set function $\nu$ on $\mathrsfs{B}$ such that, for every finite (resp.\ countable) family $\{B_j\}_{j\in J}$ of elements of $\mathrsfs{B}$, 
\begin{equation}\label{eqmax}
\nu(\bigcup_{j\in J} B_j) = \bigoplus_{j \in J} \nu(B_j). 
\end{equation}
Examples of ($\sigma$-)maxitive measures were collected in \cite[Chapter~I]{Poncet11}. 


\subsection{Reminders on the idempotent $\odot$-integral}\label{intshilkret}

The first extension of the Lebesgue integral, among many others, was proposed by Vitali \cite{Vitali25b}, who replaced $\sigma$-additive measures by some more general set functions. 
Decades later, the Choquet integral (see Choquet \cite{Choquet54} for the definition) was born, with the same idea of using ``capacities'' instead of measures. 

Inspired by Choquet, many authors have intended to replace operations $(+, \times)$, which are the basic algebraic framework of both the Lebesgue and the Choquet integrals, by some more general pair $(\dot{+}, \dot{\times})$ of binary relations on $\reels_+$ or $\overline{\reels}_+$. 
In the case where $(\dot{+}, \dot{\times})$ is the pair $(\max, \times)$, one gets the \textit{Shilkret integral} discovered by Shilkret \cite{Sugeno74}. 
If $(\dot{+}, \dot{\times})$ is the pair $(\max, \min)$, one gets the \textit{Sugeno integral} or \textit{fuzzy integral} due to Sugeno \cite{Sugeno74}. 
In the general case, one talks about the \textit{pan-integral} or \textit{seminormed fuzzy integral}, see e.g.\ Weber \cite{Weber84}, Sugeno and Murofushi \cite{Sugeno87}, Wang and Klir \cite{Wang92}, Pap \cite{Pap95, Pap02b}. 
In this paper, we shall limit our attention to the case where $\dot{+}$ is the maximum operation $\max = \oplus$ and $\dot{\times}$ is a pseudo-multiplication (i.e.\ a binary relation $\odot$ satisfying the properties given in §~\ref{par:odot}). 

A map $f : E \rightarrow \overline{\reels}_+$ is \textit{$\mathrsfs{B}$-measurable} 
if $\{ f > t \} := \{ x \in E : f(x) > t \} \in \mathrsfs{B}$, for all $t \in \reels_+$. 

\begin{definition}\cite{Sugeno87}
Let $\nu$ be a maxitive measure on $\mathrsfs{B}$, and let $f : E \rightarrow \overline{\reels}_+$ be a $\mathrsfs{B}$-measurable map. The \textit{idempotent $\odot$-integral}  
of $f$ with respect to $\nu$ is defined by 
\begin{equation}\label{defint}
\nu(f) = \dint{E} f \odot d\nu = \bigoplus_{t \in \reels_+} t \odot \nu(f > t). 
\end{equation}
\end{definition}

The occurrence of $\infty$ in the notation $\ddint \,$ is \textit{not} an integration bound, see \cite[Theorem~I-5.7]{Poncet11} for a justification.

\begin{proposition}
Let $\nu$ be a $\sigma$-maxitive measure on $\mathrsfs{B}$. 
Then, for all $\mathrsfs{B}$-measurable maps $f, g : E \rightarrow \overline{\reels}_+$, and all $r \in \reels_+$, $B \in \mathrsfs{B}$, the following properties hold: 
\begin{itemize}
	\item $\nu(1_B) = \nu(B)$, 
	\item homogeneity: $\nu(r \odot f) = r \odot \nu(f)$, 
	\item $\sigma$-maxitivity: $\nu(\bigoplus_n f_n) = \bigoplus_n \nu(f_n)$, for every sequence of $\mathrsfs{B}$-measurable maps $f_n : E \rightarrow \overline{\reels}_+$, 
	\item $B \mapsto \dint{B} f \odot d\nu$ is a $\sigma$-maxitive measure on $\mathrsfs{B}$, 
\end{itemize}
\end{proposition}

\begin{proof}
See Sugeno and Murofushi \cite[Proposition~6.1]{Sugeno87}. 
\end{proof}

Further properties of the idempotent $\odot$-integral might be found in \cite{Shilkret71}, \cite{Agbeko00}, \cite{Puhalskii01}, \cite{deCooman97}, and \cite{Poncet11} in the case where $\odot$ is the usual multiplication, i.e.\ where the Shilkret integral is considered. For the Sugeno integral, see e.g.\ \cite{Sugeno74}.

\section{The idempotent Radon--Nikodym theorem}\label{secRN} 

\subsection{Introduction}


In this section, we recall the Sugeno--Murofushi theorem, which states the existence of Radon--Nikodym derivatives for the 
idempotent $\odot$-integral \cite[Theorem~8.2]{Sugeno87}. 
Here, $\mathrsfs{B}$ still denotes a $\sigma$-algebra. 

Let $\nu$ and $\tau$ be maxitive measures on $\mathrsfs{B}$. Then $\nu$ \textit{has a density with respect to} $\tau$ if there exists some $\mathrsfs{B}$-measurable map (called \textit{density}) $c : E \rightarrow \overline{\reels}_+$ such that 
\begin{equation}\label{eq:dens}
\nu(B) = \dint{B} c \odot d\tau, 
\end{equation}
for all $B \in \mathrsfs{B}$. 
Note that, if $\nu$ has a density with respect to $\tau$, then $\nu$ is $\odot$-absolutely continuous with respect to $\tau$, according to the following definition. 

\begin{definition}\cite{Sugeno87}
Let $\nu$, $\tau$ be monotone set functions on $\mathrsfs{B}$. Then $\nu$ is \textit{$\odot$-absolutely continuous with respect to $\tau$} (or $\tau$ \textit{$\odot$-dominates} $\nu$), in symbols $\nu \ll_{\odot} \tau$, if for all $B \in \mathrsfs{B}$ such that $\tau(B)$ be $\odot$-finite, $\nu(B) \leqslant \infty \odot \tau(B)$. 
\end{definition}

Absolute continuity, although necessary in Equation~(\ref{eq:dens}), seems a priori too poor a condition for ensuring the existence of a density. 
For instance, every $\sigma$-maxitive measure $\nu$ is $\odot$-absolutely continuous with respect to the $\sigma$-maxitive measure $\delta_{\#}$, defined on the same $\sigma$-algebra $\mathrsfs{B}$ by $\delta_{\#}(B) = 1$ if $B$ is nonempty and $\delta_{\#}(\emptyset) = 0$; however, $\nu$ does not always have a density with respect to $\delta_{\#}$, and this latter measure is not $\sigma$-principal in general. 
We shall understand in §~\ref{subsec:existence} that $\odot$-absolute continuity is actually a necessary and sufficient condition for the existence of a density whenever the dominating measure is $\sigma$-$\odot$-finite and \textit{$\sigma$-principal}.

\subsection{$\odot$-Finiteness of the density} 




A set function $\nu : \mathrsfs{B} \rightarrow \overline{\reels}_+$ is \textit{semi-$\odot$-finite} if $\nu(B) = \bigoplus_{A \subset B} \nu(A)$ for all $B \in \mathrsfs{B}$, where the supremum is taken over $\{ A \in \mathrsfs{B} : A \subset B, \nu(A) \ll_{\odot} \infty \}$. 

\begin{proposition}\label{prop:semi-finite}
Let $\nu$, $\tau$ be $\sigma$-maxitive measures on $\mathrsfs{B}$. 
Assume that $\nu$ is semi-$\odot$-finite and admits a density $c$ with respect to $\tau$. Then $\nu$ admits a $\odot$-finite-valued density with respect to $\tau$. 
\end{proposition}

\begin{proof}
In the case where $\odot$ is degenerate, the fact that $\nu$ be semi-$\odot$-finite implies $\nu = 0$; then the result is clear. 
So for the rest of the proof we assume that $\odot$ is non-degenerate. 
Let $F = \{ x \in E: c(x) \ll_{\odot} \infty \}$ and let $c_1 = c \odot 1_F$. 
By Theorem~\ref{lem:phi}, $F$ is a measurable set, so $c_1$ is $\mathrsfs{B}$-measurable (and $\odot$-finite-valued). 
Let us show that $c_1$ is still a density of $\nu$ with respect to $\tau$. Let $B \in \mathrsfs{B}$. Then 
\begin{equation}\label{eq:dec2}
\nu(B) = \dint{B} c \odot d\tau = \dint{B} c_1 \odot d\tau \oplus \dint{B \cap E \setminus F} c \odot d\tau. 
\end{equation}
We first assume that $\nu(B)$ is $\odot$-finite, and we show that $\tau(B \cap E \setminus F) = 0$. With Equation~(\ref{eq:dec2}) this will imply $\nu(B) =  \dint{B} c_1 \odot d\tau$. 
Let $F_{\odot}$ denote the set of $\odot$-finite elements. 
If $F_{\odot} = [0, \infty]$, then $E \setminus F$ is empty, so that $\tau(B \cap E \setminus F)  = 0$. 
Now suppose that $F_{\odot} = [0, \phi)$.  
Since $t \odot \tau(B \cap E \setminus F \cap \{ c > t \}) \leqslant \nu(B) \ll_{\odot} \infty$ for all $t>0$, we have by Lemma~\ref{lem:inv} that $t$ is $\odot$-finite or $\tau(B \cap E \setminus F \cap \{ c > t \}) > 0$, for all $t>0$. Consequently, 
\begin{equation}\label{eq:intphi}
\dint{B \cap E \setminus F} c \odot d\tau = \phi \odot \tau(B \cap E \setminus F). 
\end{equation} 
If $s := \tau(B \cap E \setminus F) > 0$, then by Equation~(\ref{eq:intphi}) and Theorem~\ref{lem:phi} this implies $\phi > \nu(B) \geqslant \phi \odot s \geqslant \phi$, a contradiction. Thus, we have again $\tau(B \cap E \setminus F) = 0$. 

If $\nu(B)$ is $\odot$-infinite, we use the fact that $\nu$ is semi-$\odot$-finite. We get  
$$
\nu(B) = \bigoplus_{A \subset B} \nu(A) = \bigoplus_{A \subset B} \dint{A} c_1 \odot d\tau 
\leqslant \dint{B} c_1 \odot d\tau,
$$
where the supremum is taken over $\{ A \in \mathrsfs{B} : A \subset B, \nu(A) \ll_{\odot} \infty \}$, 
so that $\nu(B) = \dint{B} c_1 \odot d\tau$, for all $B \in \mathrsfs{B}$. 
\end{proof}




\subsection{Principality and existence of a density}\label{subsec:existence}

A monotone set function $\mu$ on $\mathrsfs{B}$ is \textit{$\sigma$-principal} if, for every $\sigma$-ideal $\mathrsfs{I}$ of $\mathrsfs{B}$, there exists some $L \in \mathrsfs{I}$ such that $S \setminus L$ is $\mu$-negligible, for all $S \in \mathrsfs{I}$. Proposition~\ref{prop:principal} will justify this terminology. 
Sugeno and Murofushi \cite{Sugeno87} proved a Radon--Nikodym theorem for the idempotent $\odot$-integral when the dominating measure is $\sigma$-$\odot$-finite and $\sigma$-principal. 

\begin{theorem}[Sugeno--Murofushi]\label{sugeno-murofushi}
Let $\nu$, $\tau$ be $\sigma$-maxitive measures on $\mathrsfs{B}$. 
Assume that $\tau$ is $\sigma$-$\odot$-finite and $\sigma$-principal. 
Then $\nu \ll_{\odot} \tau$ if and only if there exists some $\mathrsfs{B}$-measurable map $c : E \rightarrow \overline{\reels}_+$ such that 
$$
\nu(B) = \dint{B} c\odot d\tau, 
$$
for all $B\in \mathrsfs{B}$. 
\end{theorem}

\begin{proof}
See \cite[Theorem~8.2]{Sugeno87} 
for the original proof, and \cite[Chapter~III]{Poncet11} for an alternative proof in the case where $\odot$ is the usual multiplication. 
\end{proof}

If $\odot$ is the usual multiplication, the hypothesis of $\sigma$-$\odot$-finiteness of $\tau$ cannot be removed: consider for instance a finite set $E$, and let $\nu = \delta_{\#}$ and $\tau = \infty \cdot \delta_{\#}$ be $\sigma$-maxitive measures defined on the power set of $E$. Then $\tau$ is $\sigma$-principal and $\nu$ is absolutely continuous with respect to $\tau$, but $\nu$ never has a density with respect to $\tau$. 

After the article \cite{Sugeno87}, many authors have published results of Radon--Nikodym flavour for maxitive measures. This is the case of Agbeko \cite{Agbeko95}, Akian \cite{Akian99}, Barron et al.\ \cite{Barron00}, and Drewnowski \cite{Drewnowski09}. In some cases, the authors were not aware of the existence of the Shilkret integral. In \cite{Poncet11}, we explained why these results are already encompassed in the Sugeno--Murofushi theorem, and we also gave another proof of this theorem  with the help of order-theoretical arguments.

\section{The quotient space and the Radon--Nikodym property}\label{sec:quotient}

In this section, we characterize those $\sigma$-maxitive measures $\tau$ with the \textit{Radon--Nikodym property}, i.e.\ such that all $\sigma$-maxitive measures that are $\odot$-dominated by $\tau$ have a density with respect to $\tau$. At first, we shall introduce the quotient space associated with $\tau$. 

Let $\tau$ be a $\sigma$-maxitive measure on $\mathrsfs{B}$. On $\mathrsfs{B}$ we define an equivalence relation $\sim$ by $A \sim B$ if $A \cup N = B \cup N$, for some $\tau$-negligible subset $N$. We write $B^{\tau}$ for the equivalence class of $B \in \mathrsfs{B}$. The quotient set derived from $\sim$ is called the \textit{quotient space} associated with $\tau$, and denoted by $\mathrsfs{B}/\tau$. The quotient space can be equipped with the structure of a $\sigma$-complete lattice induced by the partial order $\leqslant$ defined by $A^{\tau} \leqslant B^{\tau}$ if $A \subset B \cup N$, for some $\tau$-negligible subset $N$. 

The next proposition, partly due to Sugeno and Murofushi, characterizes $\sigma$-principal $\sigma$-maxitive measures defined on a $\sigma$-algebra.  
A maxitive measure on $\mathrsfs{B}$ satisfies the \textit{countable chain condition} (or is \textit{CCC}) if each family of non-negligible pairwise disjoint elements of $\mathrsfs{B}$ is countable. 
(A CCC maxitive measure is sometimes called \textit{$\sigma$-decomposable}, but this terminology should be avoided, because of possible confusion with the notion of decomposability used e.g.\ by Weber \cite{Weber84}.) 

\begin{proposition}\label{prop:principal}
Let $\tau$ be a $\sigma$-maxitive measure on $\mathrsfs{B}$. The following conditions are equivalent:
\begin{enumerate}
	\item\label{principal1} $\tau$ is $\sigma$-principal, 
	\item\label{principal2} $\tau$ satisfies the countable chain condition, 
	\item\label{principal3} the quotient space $\mathrsfs{B}/\tau$ is $\sigma$-principal, in the sense that every $\sigma$-ideal of $\mathrsfs{B}/\tau$ is a principal ideal, 
	\item\label{principal4} there is some $\sigma$-principal $\sigma$-additive measure $m$ on $\mathrsfs{B}$ such that, for all $B \in \mathrsfs{B}$, $m(B) = 0 \Leftrightarrow \tau(B) = 0$. 
\end{enumerate}
\end{proposition}

\begin{proof}
(\ref{principal4}) $\Rightarrow$ (\ref{principal1}) This implication is clear. 

(\ref{principal1}) $\Rightarrow$ (\ref{principal2}) 
Assume that $\tau$ is $\sigma$-principal, and let $\mathrsfs{A}$ be a family of non-negligible pairwise disjoint elements of $\mathrsfs{B}$. Let $\mathrsfs{I}$ be the $\sigma$-ideal generated by $\mathrsfs{A}$, and let $L \in \mathrsfs{I}$ such that $\tau(I \setminus L) = 0$ for all $I \in \mathrsfs{I}$. We can choose $L$ of the form $L = \bigcup_{n \in \mathbb{N}} A_n$, with $A_n \in \mathrsfs{A}$ for all $n$. 
Now let us show that $\mathrsfs{A} = \{ A_n : n \in \mathbb{N} \}$, which will prove that $\mathrsfs{A}$ is countable. So let $A \in \mathrsfs{A}$, and assume that $A \neq A_n$ for all $n$. Then $A \cap A_n = \emptyset$ for all $n$, i.e.\ $A \subset E \setminus L$. Moreover, the definition of $L$ implies $\tau(A \setminus L) = 0$, so that $\tau(A) = 0$, a contradiction. 

(\ref{principal2}) $\Rightarrow$ (\ref{principal1}) 
This was proved by Sugeno and Murofushi \cite[Lemma~4.2]{Sugeno87} with the help of Zorn's lemma. 

(\ref{principal1}) $\Rightarrow$ (\ref{principal4}) Let $m$ be the map defined on $\mathrsfs{B}$ by
$$
m(B) = \bigoplus_{\pi} \sum_{B' \in \pi} \tau(B\cap B'),
$$
where the supremum is taken over the set of finite $\mathrsfs{B}$-partitions $\pi$ of $E$. Then $m$, called the \textit{disjoint variation} of $\tau$, is the least $\sigma$-additive measure greater than $\tau$ (see e.g.\ Pap \cite[Theorem~3.2]{Pap95}). 
Moreover, $\tau(B) > 0$ if and only if $m(B) > 0$. 
Let us show that $m$ is $\sigma$-principal. If $\mathrsfs{I}$ is a $\sigma$-ideal of $\mathrsfs{B}$, there exists some $L \in \mathrsfs{I}$ such that $\tau(B \backslash L) = 0$ for all $B \in \mathrsfs{I}$. If $B \in \mathrsfs{I}$, then $\tau(B\cap B' \backslash L) = 0$ for all $B' \in \mathrsfs{B}$, since $B\cap B' \in \mathrsfs{I}$. Hence we have $m(B \backslash L) = 0$. 

(\ref{principal3}) $\Rightarrow$ (\ref{principal1}) Let $\mathrsfs{I}$ be a $\sigma$-ideal of $\mathrsfs{B}$, and let $I = \{ B^{\tau} : B \in \mathrsfs{I} \}$. Then $I$ is closed under countable suprema, and if $A^{\tau} \leqslant B^{\tau}$ with $B \in \mathrsfs{I}$, then $A \subset B \cup N$ for some negligible subset $N \in \mathrsfs{B}$. Hence $A \cap (E\setminus N) \subset B$, so that $A \cap (E\setminus N) \in \mathrsfs{I}$. Since $A \cap (E\setminus N) \sim A$, this implies that $A^{\tau} \in I$. Thus, $I$ is a $\sigma$-ideal of $\mathrsfs{B}/\tau$. Since $\mathrsfs{B}/\tau$ is $\sigma$-principal, there is some $L \in \mathrsfs{I}$ such that $B^{\tau} \in I$ if and only if $B^{\tau} \leqslant L^{\tau}$. We deduce that $\tau(B \setminus L) = 0$ for all $B\in \mathrsfs{I}$, which proves that $\tau$ is $\sigma$-principal. 

(\ref{principal1}) $\Rightarrow$ (\ref{principal3}) Let $I$ be a $\sigma$-ideal of $\mathrsfs{B}/\tau$. Then $\mathrsfs{I} = \{ B \in \mathrsfs{B} : B^{\tau} \in I \}$ is a $\sigma$-ideal of $\mathrsfs{B}$. Since $\tau$ is $\sigma$-principal, there is some $L \in \mathrsfs{I}$ such that $\tau(B \setminus L) = 0$ for all $B \in \mathrsfs{I}$. Then $B^{\tau} \in I$ if and only if $B^{\tau} \leqslant L^{\tau}$, i.e.\ $I$ is a principal ideal. 
\end{proof}

Following Segal \cite{Segal51}, a monotone set function $\mu$ on $\mathrsfs{B}$ is \textit{localizable} if, for every $\sigma$-ideal $\mathrsfs{I}$ of $\mathrsfs{B}$, there exists some $L \in \mathrsfs{B}$ such that 
\begin{enumerate}
	\item $S \setminus L$ is $\mu$-negligible, for all $S \in \mathrsfs{I}$, 
	\item if there is some $B \in \mathrsfs{B}$ such that $S \setminus B$ is $\mu$-negligible for all $S \in \mathrsfs{I}$, then $L \setminus B$ is $\mu$-negligible. 
\end{enumerate}
In this case, $\mathrsfs{I}$ is said to be \textit{localized} in $L$. It is clear that a monotone set function is localizable if and only if the associated quotient space is a complete lattice. 
Note also that localizability is weaker than $\sigma$-principality. 



Here comes the characterization of the Radon--Nikodym property.

\begin{theorem}
Given a non-degenerate pseudo-multiplication $\odot$, a $\sigma$-max\-itive measure $\tau$ on $\mathrsfs{B}$ satisfies the Radon--Nikodym property with respect to the idempotent $\odot$-integral if and only if $\tau$ is $\sigma$-$\odot$-finite and $\sigma$-principal. 
\end{theorem}

\begin{proof}
The `if' part of this theorem is due to Sugeno and Murofushi \cite[Theorem~8.2]{Sugeno87}. 
The `only if' part 
 is proved in six steps. Let $\tau$ be a $\sigma$-maxitive measure satisfying the Radon--Nikodym property. 

\textit{Claim 1: $\tau$ is localizable. }

Let $\mathrsfs{I}$ be a $\sigma$-ideal of $\mathrsfs{B}$, and let $\nu$ be defined on $\mathrsfs{B}$ by 
$$
\nu(B) = \bigoplus_{I \in \mathrsfs{I}} \tau(B \cap I). 
$$
Then $\nu$ is a $\sigma$-maxitive measure on $\mathrsfs{B}$, $\odot$-absolutely continuous with respect to $\tau$, hence we can write 
$$
\nu(B) = \dint{B} c\odot d\tau, 
$$
for some $\mathrsfs{B}$-measurable map $c : E \rightarrow \overline{\reels}_+$. Defining $L = \{ c \neq 0 \} \in \mathrsfs{B}$, one can see that $\mathrsfs{I}$ is localized in $L$. 

\textit{Claim 2: $\tau$ is $\sigma$-principal. }

Let $\mathrsfs{I}$ be a $\sigma$-ideal in $\mathrsfs{B}$, and let $L \in \mathrsfs{B}$ localizing $\mathrsfs{I}$ with respect to $\tau$. 
We define the non-decreasing family of $\sigma$-ideals $(\mathrsfs{J}_t)_{t > 0}$ by 
$$
\mathrsfs{J}_t = \{ I \cup B : I \in \mathrsfs{I}, B \in \mathrsfs{B}, \tau(B) \leqslant t \}. 
$$
Now let $\nu$ be the map defined on $\mathrsfs{B}$ by 
$$
\nu(B) = \inf \{ t > 0 : B \in \mathrsfs{J}_t \}, 
$$
for all $B \in \mathrsfs{B}$. 
Then, by a result due to Nguyen et al.\ \cite[§~2]{Nguyen97} we know that $\nu$ is a $\sigma$-maxitive measure (see also \cite{Poncet12b}). 
Moreover, $\nu(B) \leqslant \tau(B)$ for all $B\in\mathrsfs{B}$, so that $\nu \ll_{\odot} \tau$, hence we can write 
\begin{equation}\label{eqn:intnutaubis}
\nu(\cdot) = \dint{\cdot} c\odot d\tau, 
\end{equation}
for some $\mathrsfs{B}$-measurable map $c : E \rightarrow \overline{\reels}_+$. 
If $I \in \mathrsfs{I}$, then $\nu(I) = 0$, hence, using Equation~(\ref{eqn:intnutaubis}) and the fact that $\odot$ has no zero divisors, $\tau(I \cap \{ c > t\}) = 0$, for all $t > 0$. 
This implies that $\tau(I \setminus \{ c = 0\}) = 0$, for all $I \in \mathrsfs{I}$. By definition of $L$, we deduce that $\tau(L \setminus \{ c =0\}) = 0$. Therefore, $\nu(L) = \nu(L \setminus \{ c = 0 \}) \oplus \nu(L \cap \{ c=0\}) = 0$. 
The definition of $\nu$ implies that $L \in \mathrsfs{J}_q$ for all $q \in \mathbb{Q}^*_+$. Thus, we can write $L = I_q \cup B_q$ for all $q \in \mathbb{Q}^*_+$, with $I_q \in \mathrsfs{I}$ and $\tau(B_q) \leqslant q$. Now, one can check that $L = I_0 \cup N_0$ with $I_0 := \bigcup_{q \in \mathbb{Q}^*_+} I_q$ and $N_0 := \bigcap_{q \in \mathbb{Q}^*_+} B_q$. 
Since $\tau(N_0) = 0$, it appears that we have found $I_0 \in \mathrsfs{I}$ such that $\nu(I \setminus I_0) = 0$, for all $I \in \mathrsfs{I}$, so we have proved that $\tau$ is $\sigma$-principal. 

\textit{Claim 3: $\tau$ has no $\odot$-spot. }

We define a \textit{$\odot$-spot} of $\tau$ as an element $B_0$ of $\mathrsfs{B}$ such that $\tau(B_0)$ is $\odot$-infinite and $\tau(A)$ is either zero or $\odot$-infinite for all $A \subset B_0$. 
Now assume that $\tau$ has such a $\odot$-spot $B_0$. 
If $\mathrsfs{I}$ is the $\sigma$-ideal of $\mathrsfs{B}$ generated by $\{ B \in \mathrsfs{B} : \tau(B) \ll_{\odot} \infty \}$, we define the $\sigma$-maxitive measure $\nu$ by $\nu(B) = 0$ if $B \in \mathrsfs{I}$, and $\nu(B) = 1_{\odot}$ otherwise, where $1_{\odot}$ is the left identity element of $\odot$. 
Since $\nu \ll_{\odot} \tau$, there exists some $\mathrsfs{B}$-measurable map $f : E \rightarrow \overline{\reels}_+$ such that $\nu(B) = \dint{B} f \odot d\tau$, for all $B \in \mathrsfs{B}$. Then 
\begin{equation}\label{eqn:unzerobis}
1_{\odot} = \nu(B_0) = \bigoplus_{t > 0} t \odot g(t), 
\end{equation} 
where $g(t) := \tau(B_0 \cap \{ f > t \})$. Thus, $t \odot g(t) \leqslant 1_{\odot}$, so $g(t)$ is $\odot$-finite by Lemma~\ref{lem:inv}, for all $t> 0$. This implies that $g(t) = 0$ for all $t > 0$ by definition of $B_0$, which contradicts Equation~(\ref{eqn:unzerobis}). 

\textit{Claim 4: $\tau(E) \leqslant \phi$, where $\phi$ is the supremum of the set of $\odot$-finite elements. }

We can suppose that $\phi < \infty$. Recall that, since $\odot$ is supposed to be non-degenerate, we have $\phi > 0$. Thanks to the Radon--Nikodym property, we have 
$$
\min(\tau(B), \phi) = \dint{B} c \odot d\tau, 
$$
for all $B \in \mathrsfs{B}$, for some $\mathrsfs{B}$-measurable map $c  : E \rightarrow \overline{\mathbb{R}}_+$. 
Assume that $\tau(E) > \phi$. Since $\tau(\{ c = 0 \}) = 0 < \phi$, we deduce that $\tau(E) = \tau(\{ c > 0 \}) > \phi$. This implies that $\tau(\{ c > t_0 \}) > \phi$ for some $0 < t_0 < \phi$. Thus, 
$$
\phi = \min(\tau(E), \phi) \geqslant t_0 \odot \tau(\{ c > t_0 \}) \geqslant t_0 \odot \phi.
$$ 
By Theorem~\ref{lem:phi}, $t_0 \odot \phi = \phi$, so that $t_0 \odot \tau(\{ c > t_0 \}) = \phi$. But this last identity is not possible by Corollary~\ref{lem:dec}. This contradiction shows that $\tau(E) \leqslant \phi$. 

\textit{Claim 5: $\tau$ is semi-$\odot$-finite. }

Let $\nu$ be the map defined on $\mathrsfs{B}$ by $\nu(B) = \bigoplus_{A \subset B} \tau(A)$, where the supremum is taken over $\{ A \in \mathrsfs{B} : A \subset B, \tau(A) \ll_{\odot} \infty \}$. Then $\nu$ is a $\sigma$-maxitive measure such that $\nu(B) = \tau(B)$ whenever $\tau(B)$ is $\odot$-finite. 
In particular,  $\nu \ll_{\odot} \tau$. 
Assume that $\nu(B_1) < \tau(B_1)$, for some $B_1 \in \mathrsfs{B}$. 
Let $c : E \rightarrow \overline{\reels}_+$ be a $\mathrsfs{B}$-measurable map such that Equation~(\ref{eqn:intnutaubis}) is satisfied, 
and let $A_t = B_1 \cap \{ c > t\}$. Using Claim~4, we have $\phi > \nu(B_1) \geqslant t\odot\tau(A_t)$, where $\phi$ is the supremum of the set of $\odot$-finite elements. 
So, by Lemma~\ref{lem:inv}, $\tau(A_t)$ is $\odot$-finite, for all $t > 0$. 
Moreover, since $\nu(\{c = 0 \}) = 0$ and $\tau$ has no $\odot$-spot by Claim~3, we deduce that $\tau(\{ c = 0 \})$ is $\odot$-finite. Thus, $\nu(\{ c = 0 \}) = \tau(\{ c = 0 \}) = 0$, so $\tau(B_1) = \tau(B_1 \cap \{c > 0 \}) = \bigoplus_{q \in \mathbb{Q}^*_+} \tau(A_q)$, and the definition of $\nu$ implies $\tau(B_1) \leqslant \nu(B_1)$, a contradiction. 

\textit{Claim 6: $\tau$ is $\sigma$-$\odot$-finite. }

Let $\mathrsfs{I}$ be the $\sigma$-ideal generated by all $A \in \mathrsfs{B}$ such that $\tau(A)$ be $\odot$-finite. Since $\tau$ is $\sigma$-principal, there is some $L \in \mathrsfs{I}$ such that $\tau(A \setminus L) = 0$ for all $A \in \mathrsfs{I}$. We can choose $L$ of the form $L = \bigcup_{n \geqslant 1} A_n$, with $\tau(A_n)$ $\odot$-finite for all $n$. Since $\tau$ is semi-$\odot$-finite, $\tau(B) = \tau(B \cap L)$ for all $B$. 
In particular, $\tau(E \setminus L) = 0$, so $E$ is equal to the union of the family $(A_n)_{n \in \mathbb{N}}$ with $A_0 := E\setminus L$, and $\tau(A_n)$ is $\odot$-finite for all $n$. 
This proves that $\tau$ is $\sigma$-$\odot$-finite. 
\end{proof}

\begin{corollary}
Let $\tau$ be a $\sigma$-maxitive measure on $\mathrsfs{B}$. Then $\tau$ satisfies the Radon--Nikodym property with respect to the Shilkret integral if and only if $\tau$ is $\sigma$-finite and $\sigma$-principal. 
\end{corollary}

\begin{corollary}
Let $\tau$ be a $\sigma$-maxitive measure on $\mathrsfs{B}$. Then $\tau$ satisfies the Radon--Nikodym property with respect to the Sugeno integral if and only if $\tau$ is $\sigma$-principal. 
\end{corollary}

\section{Conclusion and perspectives}\label{seccon}

In this work, we derived a converse statement to the Sugeno--Murofushi theorem, i.e.\ we characterized those $\sigma$-maxitive measures satisfying the Radon--Nikodym property with respect to the idempotent $\odot$-integral as being $\sigma$-$\odot$-finite $\sigma$-principal. 
This theorem specializes to the Shilkret (resp.\ Sugeno) integral when the binary relation $\odot$ coincides with the usual multiplication $\times$ (resp.\ the minimum $\wedge$). 
Our result does not exist in classical measure theory, at least not in such a concise and exact form. 

\begin{acknowledgements}
I am very grateful to Colas Bardavid who made very accurate suggestions. 
I also thank Marianne Akian who made some kind remarks, and Pr.\ Jimmie D.\ Lawson for his advice and comments. 
\end{acknowledgements}

\bibliographystyle{plain}

\def\cprime{$'$} \def\cprime{$'$} \def\cprime{$'$} \def\cprime{$'$}
  \def\ocirc#1{\ifmmode\setbox0=\hbox{$#1$}\dimen0=\ht0 \advance\dimen0
  by1pt\rlap{\hbox to\wd0{\hss\raise\dimen0
  \hbox{\hskip.2em$\scriptscriptstyle\circ$}\hss}}#1\else {\accent"17 #1}\fi}
  \def\ocirc#1{\ifmmode\setbox0=\hbox{$#1$}\dimen0=\ht0 \advance\dimen0
  by1pt\rlap{\hbox to\wd0{\hss\raise\dimen0
  \hbox{\hskip.2em$\scriptscriptstyle\circ$}\hss}}#1\else {\accent"17 #1}\fi}

\end{document}